\documentclass{amsart}

\usepackage[utf8]{inputenc}
\usepackage[OT2,T1]{fontenc}
\usepackage[hmargin=1in]{geometry}

\usepackage{amsbsy}
\usepackage{amscd}
\usepackage{amsmath}
\usepackage{amssymb}
\usepackage{commath}
\usepackage{mathrsfs}
\usepackage{mathtools}
\usepackage{stmaryrd}

\usepackage[matrix,arrow]{xy}
\usepackage{booktabs}
\usepackage{caption}
\usepackage{longtable}
\usepackage{tikz-cd}

\usepackage{comment}
\usepackage{enumerate}
\usepackage{lipsum}
\usepackage[normalem]{ulem}

\usepackage[table,dvipsnames]{xcolor}
\usepackage{natbib}
\usepackage{hyperref}
\usepackage[customcolors]{hf-tikz}

\hypersetup{
  colorlinks=true,
  linkcolor=DarkOrchid,
  filecolor=blue,
  citecolor=olive,
  urlcolor=orange,
  pdftitle={p-TSNC Project},
}

\definecolor{Green}{rgb}{0.0, 0.5, 0.0}

\DeclareSymbolFont{cyrletters}{OT2}{wncyr}{m}{n}
\DeclareMathSymbol{\Sha}{\mathalpha}{cyrletters}{"58}

\newtheorem{thmalph}{Theorem}[section]

\DeclareMathOperator{\Gal}{Gal}

\DeclareMathOperator{\Image}{Im}

\DeclareMathOperator{\ord}{ord}

\DeclareMathOperator{\rank}{rank}

\DeclareMathOperator{\Sel}{Sel}

\renewcommand{\C}{\mathbb{C}}

\newcommand{\N}{\mathbb{N}}
\newcommand{\Q}{\mathbb{Q}}

\newcommand{\Qp}{\mathbb{Q}_p}

\newcommand{\Z}{\mathbb{Z}}

\newcommand{\Zp}{\mathbb{Z}_p}

\newcommand{\cK}{\mathcal K}
\newcommand{\cL}{\mathcal{L}}
\newcommand{\cO}{\mathcal O}

\theoremstyle{plain}
\newtheorem*{theorem*}{Theorem}
\newtheorem*{conj*}{Conjecture}
\newtheorem*{Ques*}{Question}
\newtheorem*{cor*}{Corollary}
\newtheorem{theorem}{Theorem}[section]
\newtheorem{lemma}[theorem]{Lemma}
\newtheorem{Proposition}[theorem]{Proposition}

\theoremstyle{definition}
\newtheorem{definition}[theorem]{Definition}

\theoremstyle{hyp}

\theoremstyle{remark}
\newtheorem{remark}[theorem]{Remark}
\newtheorem*{Remark*}{Remark}

\makeatletter
\theoremstyle{plain}
\newtheorem*{intr@thm}{\intr@thmname}

\newtheorem*{c@njecture}{\conjn@name}
\newcommand{\myl@bel}[2]{%
  \protected@write\@auxout{}{%
    \string\newlabel{#1}{{#2}{\thepage}{#2}{#1}{}}%
  }%
  \hypertarget{#1}{}%
}

\makeatother

\begin{document}
\title[]{Iwasawa invariants of Bertolini--Darmon theta elements}

\author[Abhishek]{Abhishek}
\address[Abhishek]{Harish-Chandra Research Institute, Chhatnag Road, Jhunsi, Prayagraj  211 019 India.}
\email{abhi.math04@gmail.com, abhishek@hri.res.in}
\author[Ray]{Jishnu Ray}
\address[Ray]{
	 Harish-Chandra Research Institute, Chhatnag Road, Jhunsi, Prayagraj 211 019, India.}
\address[Ray]{
     Homi Bhabha National Institute, Training School Complex, Anushakti Nagar, Mumbai 400 094, India.
}
\email{jishnuray1992@gmail.com, jishnuray@hri.res.in}
\author[Pronay Kumar Karmakar]{Pronay Kumar Karmakar}
\address[Karmakar]{Harish-Chandra Research Institute,  Chhatnag Road, Jhunsi, Prayagraj  211 019 India.}
\email{kumarpro9math@gmail.com, pronaykarmakar@hri.res.in}



\begin{abstract}
In this article we study the Iwasawa invariants of Bertolini--Darmon theta elements in the anticyclotomic $\Z_p$-extension of an imaginary quadratic field $K$ for weight two modular forms $f\in S_2(\Gamma_0(N))$. We cover both the cases of ordinary and non-ordinary reduction at a prime $p$. Our results extend the known results of Pollack--Weston and Gajek-Leonard--Lei in the cyclotomic setting.  
\end{abstract}

\maketitle

\section{Introduction}
Fix a prime $p\ge 5$ and $f=\sum_{n=1}^{\infty} a_n(f) q^n$ be a cuspidal newform of weight $k\ge 2$ and level $\Gamma_0(N)$ with $(N, p)=1$. Throughout this article we fix an embedding $\overline{\Q} \hookrightarrow \overline{\Q_p}$. Let $\cK$ be the finite extension of $\Qp$  generated by the Fourier coefficients of $f$ and $\cO$ be its ring of integers with a fixed uniformizer $\varpi$. In \cite{PW}, Pollack and Weston studied the Iwasawa invariants of Mazur-Tate elements denoted by $\theta_n\in \Z_p[\Gal(\Q(\mu_{p^n})/\Q)]$ attached to $f$ in the cyclotomic $\Zp$-extension of $\Q$. 

When $f$ is ordinary at $p$, it is shown in \cite[Proposition~3.7]{PW} that the Iwasawa invariants of $\theta_n(f)$ as $n\rightarrow \infty$ are related to the Iwasawa invariants of the $p$-adic $L$-function attached to $f$. When $f$ is non-ordinary at $p$ and $k\ge 2$, the Iwasawa invariants of $\theta_n(f)$ can be related to those of Pollack’s plus and minus $p$-adic $L$-functions in the case $a_p(f)=0$, and to the Iwasawa invariants of the $\sharp/\flat$ $p$-adic $L$-functions $L_p^{\sharp/\flat}(f,K)$, defined in \cite{spr17} in the case $a_p(f)\neq 0$ (see \cite{gajek-lei,gajek}).

For general weight $k\ge 2$ non-ordinary modular forms at $p$ satisfying the Fontaine--Laffaille condition (that is, $p>k-1$), Gajek-Leonard--Lei show in \cite{gajek-lei} that the Iwasawa invariants of $\theta_n(f)$ as $n\rightarrow \infty$ are related to the Iwasawa invariants of the $\sharp/\flat$ $p$-adic $L$-functions attached to $f$. More precisely, their result is the following. 

\begin{theorem} \cite[Theorem~A]{gajek-lei} \label{thm:quote}
Let $f$ be a cuspidal newform of weight $k\ge 2$ and level $N$ with $(N, p)=1$. Assume that $f$ is non-ordinary at $p$. If $p>k-1$ and $\mu(L_p^\sharp(f))=\mu(L_p^\flat(f))\ne\infty$, then for sufficiently large $n$, we have
\begin{equation*}
    \mu(\theta_n(f))=\mu(L_p^\ast(f)),
\end{equation*}and
\begin{equation*}
    \lambda(\theta_n(f))=(k-1)q_n + \lambda(L_p^\ast(f)).
\end{equation*}
where $\ast=\flat$ if $n$ is odd and $\ast=\sharp$ if $n$ is even.
\end{theorem}
Here, \begin{equation*}
    q_n=\begin{cases}p^{n-1}-p^{n-2}+\cdots +p-1 & \text{if } n \text{ is even}, \\
p^{n-1}-p^{n-2}+\cdots +p^2-p & \text{if } n \text{ is odd}.
\end{cases}
\end{equation*} 
and $q_0=q_1=0.$

In this article, we study the Iwasawa invariants of certain $p$-adic $L$-functions over the anticyclotomic $\Zp$-extension of an imaginary quadratic field. The analogues of Mazur-Tate elements in this setting are the Bertolini--Darmon theta elements defined in \cite{Ber-Dar05} and \cite{Dar-Iov08}. 

Let $f\in S_2(\Gamma_0(N))$ be a newform of weight $2$ and level $N$. Let $K$ be an imaginary quadratic field with odd discriminant satisfying
\begin{equation}\label{disc}
\tag{disc} -D_K<-4  \text{ and } (D_K, pN)=1
\end{equation}
 We write $N=N^+N^-$, where $N^+$ is divisible only by primes which are split in $K$ and $N^-$ by primes which are inert in $K$. We make the following assumption.
\begin{equation}\label{def}
\tag{def} N^-  \text{is square-free and  product of odd number of primes}
\end{equation}
 Let $K_\infty$ denote the anticyclotomic $\Zp$-extension of $K$ and $K^{(n)}$ be the subextension of $K_\infty$ such that $[K^{(n)}:K]=p^n$. Let $\cL_{f,n}\in \mathcal{O}[\Gal(K^{(n)}/K)]$ be the Bertolini--Darmon theta elements attached to $f$, see section \ref{BD_ele}. 

In the first part of this article, we consider the case when $f$ is ordinary at $p$. Following \cite{Kim26}, the anticyclotomic $p$-adic $L$-function $L_p(f, K_\infty)$ attached to $f$ is defined as the limit of the $p$-stabilized Bertolini--Darmon theta elements (see section \ref{ord-case}). Let $\mu(\cL_{f,n})$ and $\lambda(\cL_{f,n})$ are the Iwasawa invariants of $\cL_{f,n}$ defined in section \ref{BD_ele}. We first prove an analogue of \cite[Proposition~3.7]{PW} in the anticyclotomic setting. The result in this direction is as follows.  

\begin{thmalph}[Theorem~\ref{main-thm-ord}]\label{main-thm-ordinary-intro}
Let $f\in S_2(\Gamma_0(N))$ be a newform with good ordinary reduction at $p$. Assume the hypotheses (\ref{disc}) and (\ref{def}).  Suppose that the residual Galois representation 
    $\overline{\rho}_f$ attached to $f$ is irreducible.
Then, for all sufficiently large $n$, the following hold:
\begin{enumerate}
    \item $\mu(\mathcal{L}_{f,n}) = 0$;
    \item $2\lambda(\mathcal{L}_{f,n}) = \lambda\big(L_p(f/K_\infty)\big)$.
\end{enumerate}
\end{thmalph}

In the second part of this article, we consider the case when $f$ is non-ordinary at $p$. In this setting, we consider the two cases:
\subsection{$p$ is split in $K$:} In this setting, let $L_p^{\sharp/\flat}(f, K_\infty)$ be the integral $\sharp/\flat$ $p$-adic $L$-functions attached to $f$ as in section \ref{non-ord}. The second main result of this article is an analogue of Theorem \ref{thm:quote} in the anticyclotomic setting. Our result in this direction is as follows.

\begin{thmalph} \label{thm:B}[Theorem~\ref{main-thm-2}]\label{main-thm-non-ord-intro}
  Let $f\in S_2(\Gamma_0(N))$ be a newform with $a_p(f)\equiv0\mod{\varpi}$. Assume the hypotheses (\ref{disc}), (\ref{def})  and suppose that $p$ splits in $K$. Furthermore, assume that $\mu(L_p^\sharp(f, K_\infty))=\mu(L_p^\flat(f, K_\infty))\ne\infty$. 
  Then for all sufficiently large $n$, we have 
    \begin{enumerate}
        \item $\mu(L_p^\ast(f, K))=2\mu(\mathcal{L}_{f,n})$,
        \item $2\lambda(\mathcal{L}_{f,n})=2q_n+\lambda(L_p^\ast(f,K))$
    \end{enumerate}
where $\ast = \sharp$ if $n$ is even and $\ast = \flat$ if $n$ is odd.  
 \end{thmalph}
\begin{remark}
    If $a_p=0$, we know from \cite[Theorem 2.5]{PW2011} that $\mu(L_p^\sharp(f, K_\infty))=\mu(L_p^\flat(f, K_\infty))=0$.
\end{remark}

\subsection{$p$ is inert in $K$:} In this case, we assume that $f\in S_2(\Gamma_0(N))$ is coming from an elliptic curve with supersingular reduction at $p$. Let $L_p^\pm(f, K)$ be the $p$-adic $L$-function as in section \ref{inert_sec}.  As $p\ge 5$, we have that $a_p(f)=0$. The main result in this is the following:
\begin{thmalph}\label{thm:C}
Let $f\in S_2(\Gamma_0(N))$ be a newform with $a_p(f)=0$. Assume the hypotheses (\ref{disc}), (\ref{def}) and suppose that $p$ is inert in $K$. Furthermore, assume that $\mu(L_p^+(f, K_\infty))=\mu(L_p^-(f, K_\infty))<\infty$.  Then for sufficiently large $n$, we have the following: 
\[\mu(L_p^\pm(f, K))=2\mu(\mathcal{L}_{f,n})
 \text{ and }
2\lambda(\mathcal{L}_{f,n})=\begin{cases}
  2(q_n+1)+\lambda(L_p^+(f,K)) & \text{if $n$ is even}, \\
  2q_n+\lambda(L_p^-(f,K))& \text{if $n$ is odd}. 
\end{cases}
\]

 \end{thmalph}

 The idea of the proof is as follows. We first show that the  $\sharp/\flat$ $p$-adic $L$-functions  $L_p^{\sharp/\flat}(f, K_\infty)\mod{\omega_n}$ can be written in  terms of Bertolini--Darmon theta elements using Sprung type matrices, see Lemma \ref{lm:Hfn}. Then, using Lemma \ref{l3} and Lemma \ref{lm:3.5}, we show that the Iwasawa invariants of $L_p^{\sharp/\flat}(f, K_\infty)$ are related to the Iwasawa invariants of $\cL_{f,n}$ for sufficiently large $n$.

\begin{remark}
In this short remark we will note an interesting analogy of these growth factors ($2q_n$ in Theorem \ref{thm:B} if $p$ splits in $K$ and $2(q_n+1), 2q_n$ accordingly when $n$ is even or odd in Theorem \ref{thm:C} if $p$ is inert in $K$) appears exactly in the growth of Tate-Shafarevich groups along anticyclotomic tower.

\vspace{.2cm}

    Suppose $E/\Q$ is an elliptic curve with $p$ an odd supersingular prime for $E$ with $a_p(E)=0$. Suppose $p$ splits in the imaginary quadratic field $K$ and $p \nmid h_K$. Then, under certain restrictive hypotheses (which imply that  both the plus and minus Selmer groups have trivial Iwasawa invariants),  which are 
    \begin{enumerate}
        \item $p \nmid \mathrm{Tam}(E/K)$; here $\mathrm{Tam}(E/K)$ is the Tamagawa factor of $E$ over $K$,
        \item $\Sha(E/K)[p]=0$,
        \item $E(K)/pE(K)=0$,
    \end{enumerate}
     Iovita--Pollack \cite[Theorem 5.1]{Iovita_Pollack06} showed that 
     \begin{equation*}\label{ip}
\mathrm{ord}_p(\#\Sha(E/K^{(n)})[p^\infty])-\mathrm{ord}_p(\#\Sha(E/K^{(n-1)})[p^\infty])=2q_n. 
     \end{equation*}
 In a more general setup, when the prime $p$ is inert in $K$, under the assumptions that 
 \begin{enumerate}
     \item the dual signed Selmer groups $\Sel^\pm(E/K_\infty)$ are cotorsion Iwasawa modules,
     \item $a_p(E)=0$ and  $p \nmid h_K$, 
     \item $\Sha(E/K^{(n)})[p^\infty]$ is finite for all $n \geq 0$,
 \end{enumerate}
 I\c{s}ik--Lei \cite[Theorem A]{IsikLei25} has proved that 
 \begin{equation*}\label{il}
\mathrm{ord}_p(\#\Sha(E/K^{(n)})[p^\infty])-\mathrm{ord}_p(\#\Sha(E/K^{(n-1)})[p^\infty]) = 
\begin{cases} 
2q_n +\lambda_{-} +\phi(p^n)\mu_{-}-r_\infty& \text{if } n \text{ is odd,} \\
2(q_n+1) +\lambda_{+} +\phi(p^n)\mu_{+}-r_\infty & \text{if } n \text{ is even.}
\end{cases}
\end{equation*}
 Here $\lambda_\pm$ and $\mu_\pm$ are the Iwasawa invariants of the Pontryagin dual of $\Sel^\pm(E/K_\infty)$ and $r_\infty$ is a constant defined as $\mathrm{sup}_{n \geq 0}\{\rank_{\Z}E(K^{(n)})\}.$
\end{remark}

\section*{Acknowledgment} The first author gratefully acknowledges the support received
from the HRI postdoctoral fellowship. The last two authors gratefully acknowledge support from the Inspire Research Grant, Department of Science and Technology, Government of India.


\section{Iwasawa invariants in finite-level group algebras}\label{Iwasawa-invar}
In this section, we recall the definitions and some properties of the Iwasawa invariants of elements of finite-level group algebras, which will be used in later sections.

Let $K_\infty$ be a $\Z_p$-extension of $K$ with Galois group $G\cong \Zp$, and $K^{(n)}$ be the subextension of $K_\infty$ with Galois group $G_n$ of order $p^n$.
Let us fix a finite integrally closed extension $\mathcal{O}$ of $\Z_p$ and set $\Lambda_n=\cO[G_n]$.   Let $\Lambda=\varprojlim \mathcal{O}[G_n]$ denote the Iwasawa algebra. Note that  $\Lambda\cong \mathcal{O}[[T]]$ via the map $\gamma\mapsto 1+T$, where $\gamma$ is a topological generated of $G\cong \Zp$. Let $\varpi$ be a fixed uniformizer of $\cO$ and for any integer $n\ge 1$, define $\omega_n:=(1+T)^{p^n}-1\in\cO[[T]]$.

For any $P\in \Lambda$, write $P=\sum_{j=0}^{\infty} b_jT^j$.  The Iwasawa invariants of $P$ are defined as follows: 
\begin{equation*}
    \mu(P)=\min_j \ord_p(b_j),
\end{equation*}
\begin{equation*}
      \lambda(P)=\min \{j:\ord_p(b_j)=\mu(P)\}.
\end{equation*}
It is to be noted that the above definitions are independent of the choice of isomorphism $\Lambda\cong \mathcal{O}[[T]]$.\\
We now define the Iwasawa invariants in the finite-level group algebras $\mathcal{O}[G_n]$. Suppose $\theta\in \mathcal{O}[G_n]$ such that $\theta=\sum_{\sigma\in G_n} c_{\sigma}\sigma$, we then define 
\begin{equation*}
    \mu(\theta)=\min_{\sigma\in G_n} \ord_p(c_\sigma).
\end{equation*}
To define the $\lambda$-invariants, let us first fix a uniformizer $\pi$ of $\mathcal{O}$, and set $\theta'=\pi^{-\mu(\theta)}\theta$. Let $\Bar{\theta'}$ denote the non-zero image of $\theta'$ under the natural map $\mathcal{O}[G_n]\longrightarrow \mathbb{F}[G_n]$, where $\mathbb{F}$ is the residue field of $\mathcal{O}$.\\
Since all ideals of $\mathbb{F}[G_n]$ are of the form $I_n^j$, where $I_n$ is the augmentation ideal. We then define 
\begin{equation*}
    \lambda(\theta)=\ord_{I_n}\Bar{\theta'}=\max \{j:\Bar{\theta'}\in I_n^j \}.
\end{equation*}

We now state and prove some lemmas following the work of \cite{P05} and \cite{PW}. 

\begin{lemma}\label{l1}
    For $f,g\in \Lambda_n$ we have
    \begin{enumerate}
        \item $\mu(fg)\geq \mu(f)+\mu(g)$,
        \item $\lambda(fg)=\lambda(f)+\lambda(g)$ if $\mu(fg)=0$.
    \end{enumerate}
\end{lemma}
\begin{proof}
  \noindent  \begin{enumerate}
        \item By the definition of the $\mu$-invariant we have, $f\in p^{\mu(f)}\Lambda_n$ and $g\in p^{\mu(g)}\Lambda_n$. This implies $fg\in p^{\mu(f)+\mu(g)}\Lambda_n$. This proves the first part.
    \item  It is clear that $\mu(f)=\mu(g)=0$, since $\mu(fg)=0$. Therefore the $\lambda$-invariants of $f, g$ and $fg$ are defined.  Now we have 
    \[\lambda(fg)=\ord_{I_n}(\overline{fg})=\ord_{I_n}(\bar{f})+\ord_{I_n}(\bar{g})=\lambda(f)+\lambda(g). \]
    
    \end{enumerate}

\end{proof}
Let $\pi_{n-1}^n:\cO[G_n]\longrightarrow \cO[G_{n-1}]$ be the natural projection map and  let $\nu_{n-1}^n: \cO[G_n]\rightarrow \cO[G_{n-1}]$ denote the corestriction map defined as
\[
\nu_{n-1}^n(\sigma):=\sum_{\underset{\tau\in G_n}{\tau\mapsto \sigma}}\tau, \,\,\,\,\,\, \text{ for } \sigma\in G_n. 
\]
We also define $\xi_n$ as follows;  $$\xi_n:=\sum_{\underset{\sigma^p=1}{\sigma\in G_n}}\sigma \in \cO[G_n].$$  
\begin{lemma}\label{l2}
    For $f\in \Lambda_{n-1}$ and $g\in \Lambda_n$ we have
    \begin{enumerate}
         \item $\Image(\nu_{n-1}^n)={\xi}_n\Lambda_n$, 
        \item $\pi_n^{n-1}(\nu_{n-1}^n(f))=p\cdot f$,
        \item $\nu_{n-1}^n(\pi_n^{n-1}(g))={\xi}_n\cdot g$.
    \end{enumerate}
\end{lemma}
\begin{proof} See  \cite[Lemma. 4.6]{P05}.
\end{proof}

\begin{lemma}\label{xi}
    We have $\mu(\xi_n)=0$ and $\lambda(\xi_n)=p^n-p^{n-1}$.
\end{lemma}
\begin{proof}
    Suppose $\gamma$ is a generator of $G_n$,  then $I_n$ is generated by $\gamma-1$. Therefore we have 
    \begin{equation*}
        \xi_n=\sum_{\sigma^p=1}\sigma=\sum_{a=0}^{p-1}\gamma^{ap^{n-1}}=\gamma^{p^n}-1/\gamma^{p^{n-1}}-1\equiv(\gamma-1)^{p^n-p^{n-1}} \mod p.
    \end{equation*}
    This shows $\mu(\xi_n)=0$ and $\lambda(\xi_n)=p^n-p^{n-1}$.
\end{proof}

\begin{lemma}\label{l3}
    Fix $P\in \Lambda$ and let $P_n$ denote the image of $P$ in $\mathcal{O}[G_n]$. Then for $n\gg 0$, we have $\mu(P)=\mu(P_n)$ and $\lambda(P)=\lambda(P_n)$.
\end{lemma}
\begin{proof} See \cite[Lemma. 3.1]{PW}.

\end{proof}

\begin{lemma}\label{lm:3.5}
   Let $P\in \Lambda$ and $Q\in \cO[T]$, and suppose that $P\equiv Q\mod{(\varpi, \omega_n)}$. If $\deg(Q)<p^n$, $\lambda(P)<p^n$ for $n\ge 0$, and $\mu(P)=0$,  then the Iwasawa invariants of $P$ and $Q$ agrees. 
\end{lemma}
\begin{proof} See \cite[Lemma~2.8]{gajek-lei}.

\end{proof}

\begin{lemma}\label{l4}
    For $\theta\in \mathcal{O}[G_{n-1}]$ and $g\in\mathcal{O}[G_{n}]$ we have 
    \begin{enumerate}
        \item $\mu(\pi_n^{n-1}(g))\geq \mu(g)$,
        \item $\mu(\xi_n(\theta))=\mu(\theta)$,
        \item $\lambda(\xi_n(\theta))=p^n-p^{n-1}+\lambda(\theta)$.
    \end{enumerate}
\end{lemma}

\begin{proof} The first part is clear from the definition. For the second part suppose $\theta=p^{\mu(\theta)}\theta'$ with $\mu(\theta')=0$. This implies $\xi_n(\theta)=p^{\mu(\theta)}\xi_n(\theta')$. Thus $\mu(\xi_n(\theta))=\mu(\theta)$ if $\mu(\xi_n(\theta'))=0$. This is reduced to the case where $\mu(\theta)=0$. Now let us choose $g\in \mathcal{O}[G_n]$ such that $\pi_n^{n-1}(g)=\theta$. It is to be noted that $\mu(g)=0$ by $(1)$. Thus we have 
\begin{equation*}
    \xi_n(\theta)=\xi_n(\pi_n^{n-1}(g))=\xi_n\cdot g
\end{equation*}
    by lemma \ref{l2}. So we have 
    $\mu(\xi_n(\theta))=\mu(\xi_n\cdot g)=\mu(g)=0=\mu(\theta)$.
    For the proof of $(3)$, we may assume that $\mu(\theta)=0$. Then let us choose $g\in \Lambda_n$ lifting $\theta$ and thus we have $\lambda(\xi_n(\theta))=\lambda(\xi_n\cdot g)=\lambda(g)+\lambda(\xi_n)$, by $(2)$ of Lemma \ref{l1}. This implies 
    \begin{equation*}
        \lambda(g)+\lambda(\xi_n)=p^n-p^{n-1}+\lambda(\pi_n^{n-1}(\theta))=p^n-p^{n-1}+\lambda(\theta).
    \end{equation*}
\end{proof}

\begin{lemma}\label{l5}
    Fix $\theta\in \mathcal{O}[G_n]$.
    \begin{enumerate}
        \item If $\mu(\pi_n^{n-1}(\theta))=0$, then $\mu(\theta)=0$.
        \item If $\mu(\theta)=\mu(\pi_n^{n-1}(\theta))$, then $\lambda(\pi_n^{n-1}(\theta))=\lambda(\theta)$.
    \end{enumerate}
\end{lemma}
\begin{proof} Part $(1)$ follows from the fact that $\mu(\pi_n^{n-1}(\theta))\geq \mu(\theta)$. For $(2)$ we know $\theta\mod p\in I_n^a$ if and only if $\pi_n^{n-1}(\theta)\mod p\in I_{n-1}^a$, since these augmentation ideals are principal. This implies $\lambda(\pi_n^{n-1}(\theta))=\lambda(\theta)$ since the $\mu$-invariants of both these elements are same.
    
\end{proof}


\section{Gross points and Bertolini--Darmon theta elements }\label{BD_ele}
In this section we recall the definition of Bertolini--Darmon theta elements as defined in \cite{Ber-Dar05} and \cite{Dar-Iov08}. We follow the notation and setup of \cite{BBL22}.  In the following we will first take the  approach of defining these theta elements using geometric Gross points and in section \ref{sec:Gross} we will recall the approach using explicit Gross point; both of these approaches are equivalent.

Let $p\ge 5$ be a fixed prime and $K$ be an imaginary quadratic number field.  Let $N$ be a positive integer coprime to $p$, write $N=N^+N^-$ where $N^+$ (resp. $N^-$) is divisible only by primes that are split (resp. are inert) in $K$. Throughout this section we assume that $N^-$ is square-free and is a product of an odd number of primes. 

Let $B$ be the definite quaternion algebra of discriminant $N^-$. Let $R$ be an Eichler $\Z[1/p]$-order of level $N^+$ in $B$ and $B_p:=B\otimes\Q_p$. We have a fixed isomorphism $i_p:B_p\longrightarrow M_2(\Q_p)$. 

We denote by $\mathcal{T}$ the Bruhat-Tits tree of $B_p^\times/\Q_p^\times$. Let $\mathcal{V}(\mathcal{T})$ be the set of vertices and $\vec{\mathcal{E}}(\mathcal{T})$ denotes the ordered edges of $\mathcal{T}$. Set $\mathbf{\Gamma}=R^\times/\Z[1/p]^\times$ and let $Z$ be a ring. Recall that the $Z$-valued weight two modular form on $\mathcal{T}/\mathbf{\Gamma}$ is a $Z$-valued function $h$ on $\vec{\mathcal{E}}(\mathcal{T})$ such that 
\[ 
h(\gamma e)=h(e) \qquad \text{for all } \gamma\in\mathbf{\Gamma} .
\]
We denote by $S_2(\mathcal{T}/\mathbf{\Gamma};Z)$ the space of $Z$-valued weight two modular forms on $\mathcal{T}/\mathbf{\Gamma}$. The following Proposition gives the correspondence between modular forms in $S_2(\mathcal{T}/\mathbf{\Gamma};Z)$ and $S_2(\Gamma_0(N), \C)$  as a consequence of Jacquet--Langlands correspondence, see \cite[Proposition~1.3]{Ber-Dar05} and \cite[Theorem~2.2, Proposition~2.3]{Dar-Iov08}.

\begin{Proposition}\label{prop:correspondence}
    Let $f\in S_2(\Gamma_0(N))$ be a newform, then there exists $h\in S_2(\mathcal{T}/\mathbf{\Gamma};\Z)$ such that $f$ and $h$ have the same Hecke eigenvalues at all primes $\ell\nmid N$. This form $h$ is unique up to multiplication by a non-zero scalar.
\end{Proposition}

  Let $\cK$ be the number field generated by the Fourier coefficients of $f$. We fix a prime $v$ of $\cK$ lying above $p$ and let $L$ be the completion of $\cK$ at $v$. Let $\cO$ be the ring of integers of $L$ and $\varpi$ be a fixed uniformizer of $\cO$. We now fix a newform $f\in S_2(\Gamma_0(N))$ of level $N$ and weight $2$. Using the above proposition, we can identify $f$ with an element $h\in S_2(\mathcal{T}/\mathbf{\Gamma};\cO)$, which is not divisible by $\varpi$. 

Fix an embedding $\Psi:K\longrightarrow B$ such that $\Psi(K)\cap R=\Psi(\cO_K[1/p]^\times)$. The group $\Pi_\infty:=K_p^\times/\Q_p^\times$ acts on $\mathcal{T}$ via $$\sigma\star x:=(i_p\circ\Psi)(\sigma)x$$ for any $\sigma\in\Pi_\infty$ and any vertex or edge $x$ of $\mathcal{T}$. Set $\widetilde{G}_\infty:=\Pi_\infty/v_p^{\mathbb{Z}}$, where $v_p$ is a fundamental unit of $K$. There is natural filtration of $\Pi_\infty$ as given in \cite[(2.2) and (2.3)]{Dar-Iov08}, 
\[
\cdots\subset U_n\subset \cdots \subset U_1\subset U_0\subset\Pi_\infty,
\]

\begin{remark}
We note that, if $p$ is inert in $K$ then the group $\Pi_\infty$ is compact. In this case, we have $U_0=\Pi_\infty$. On the other hand, if $p$ is split in $K$, then $\Pi_\infty$ is not compact and we have $U_0\subsetneq \Pi_\infty$, here $U_0$ is the maximal compact subgroup of $\Pi_\infty$, see \cite[\S~2.2]{Dar-Iov08}.
\end{remark} 
Let $\widetilde{G}_m:=\widetilde{G}_\infty/U_m.$ We choose a sequence of consecutive vertices $v_0, v_1, v_2, \cdots$ of $\mathcal{T}$ with coherent orientation and without backtracking such that $Stab(v_n)=U_n$ for all $n\ge 0$ (see, \cite[\S~3]{Kim19}). Each vertex $v_n$ is called the a \textit{(classical) Gross point} of level $0$ and the directed edge $e_n (=v_n\rightarrow v_{n+1})$ is called a \textit{(classical) Gross point} of level $1$.

For any $h\in S_2(\mathcal{T}/\mathbf{\Gamma};\cO_L)$, there is a sequence of functions $h_{K, m}:\widetilde{G}_m\longrightarrow \cO_L$ defined by $\sigma\mapsto h(\sigma\star v_m)$, where $v_m\in \mathcal{T}$ is choosen as in \cite{Dar-Iov08}. We now define the Bertolini--Darmon theta elements $\widetilde{\cL}_{f,m}\in \cO_L[\widetilde{G}_m]$ as follows: 
\begin{equation}\label{deh_bd_elt}
    \widetilde{\cL}_{f,m}:=\sum_{\sigma\in \widetilde{G}_m} h_{K, m}(\sigma)\cdot\sigma^{-1}.
\end{equation}
Let $\pi_{n}^{n+1}:\cO_L[\widetilde{G}_{n+1}]\longrightarrow\cO_L[\widetilde{G}_{n}]$ be the natural projection and $\widetilde{\xi}_n:\cO_L[\widetilde{G}_n]\longrightarrow\cO_L[\widetilde{G}_{n+1}]$ be the norm map as in \S \ref{Iwasawa-invar}. Then we have the following three term relation as given in \cite[Lemma~2.6]{Dar-Iov08}:
\begin{equation}\label{three-term-relation}
\pi_{n}^{n+1}(\widetilde{\cL}_{f,n+1})=a_p(f)\cdot\widetilde{\cL}_{f,n}-\widetilde{\xi}_{n-1}(\widetilde{\cL}_{f,n-1}).
\end{equation}
We define  $G_n:=\widetilde{G}_{n+1}/ \Delta$, where $\Delta$ is torsion subgroup of $\tilde{G}_{n+1}$. From the class field theory, the group $G_n$ is identified with $\Gal(K^{(n)}/K)$.
We denote $\cL_{f, n}$ the image of $\widetilde{\cL}_{f, n}$ under the natural projection $\cO_L[\widetilde{G}_{n+1}]\longrightarrow \cO_L[G_n]$. The elements $\cL_{f, n}$ are called the Bertolini--Darmon theta elements attached to $f$ over $K^{(n)}$. By a slight abuse of notation, we also denote the projection map $\cO_L[G_{n+1}]\longrightarrow \cO_L[G_n]$ by $\pi_{n}^{n+1}$ and the norm map $\cO_L[G_n]\longrightarrow \cO_L[G_{n+1}]$ by $\xi_n$. Then the three term relation in terms of $\cL_{f, n}$ is as follows:
\begin{equation}\label{three-term}
\pi_{n}^{n+1}(\cL_{f,n+1})=a_p(f)\cdot\cL_{f,n}-\xi_{n-1} (\cL_{f,n-1}).
\end{equation}
Let $\alpha$ and $\beta$ are roots of the Hecke polynomial $X^2-a_p(h)X+p$. For any $\lambda\in \set{\alpha, \beta}$, the \emph{$\lambda$-stabilized Bertolini--Darmon element} is defined by
\[\cL_{h, n}^\lambda:=\frac{1}{\lambda^{n+1}}\left(\cL_{h, n}-\frac{1}{\lambda}\xi_{n-1}\cL_{h, n-1}\right). \]
 From the three-term relation \eqref{three-term} for $\cL_{h, n}$, it follows that $\pi_n^{n+1}(\cL_{h, n+1}^\lambda)=\cL_{h, n}^\lambda$ for $n\geq 0$. The sequence $\{\cL_{h, n}^\lambda\}_{n\geq 0}$ is compatible with the natural projections $\pi_n^{n+1}$. This sequence converges to an element $\cL_h^\lambda=\varprojlim \cL_{h, n}^\lambda\in \mathcal{H}[\Gamma]$, where $\mathcal{H}(\Gamma)$ denotes the set of power series in $\cK[[T]]$ that converges in the open unit disk.  We will use these $\lambda$-stabilized Bertolini--Darmon elements to define the $p$-adic $L$-function attached to $f$ in the anticyclotomic setting when $f$ is ordinary at $p$.
\subsection{Explicit Gross points}\label{sec:Gross}
Let $K$ be an imaginary quadratic number field with odd discriminant $-D_K<-4$ and define
$$\vartheta := 
\left \lbrace
    \begin{array}{ll}
     \dfrac{ D_K - \sqrt{-D_K} }{ 2 }  & \textrm{ if } 2 \nmid D_K \\ 
\dfrac{ D_K - 2\sqrt{-D_K} }{ 4 }  & \textrm{ if } 2 \mid D_K
    \end{array}
    \right.$$
Therefore $\cO_K=\Z+\Z\vartheta$. Let $\set{1, J}$ be a $K$-basis of $B$ so that $B=K\oplus K\cdot J$ such that 
\begin{enumerate}
\item $\beta := J^2 \in \mathbb{Q}^\times$ with $\beta <0$,
\item $J \cdot t = \overline{t} \cdot J$ for all $t \in K$,
\item $\beta \in \left( \mathbb{Z}^\times_q \right)^2$ for all $q \mid pN^+$,
\item $\beta \in \mathbb{Z}^\times_q$ for all $q \mid D_K$.
\end{enumerate}
For $q\mid N^+p$, let $i_q:B_q\rightarrow M_2(\Q_q)$ be the isomorphism defined in \cite[\S~2.1]{Kim19}. 
We define the local Gross point $\varsigma_q \in B^\times_q$ for any rational prime $q$.
\begin{enumerate}
    \item Let $q \nmid N^+p$ be a prime. Then $\varsigma_q := 1$ in $B^\times_q$.
\item Let $q \mid N^+$ be a prime, and write $q = \mathfrak{q} \overline{\mathfrak{q}}$ in $\mathcal{O}_K$.
Then
$$\varsigma_q := \frac{1}{\sqrt{D_K}}\cdot \left( \begin{matrix}
\vartheta & \overline{\vartheta} \\
1 & 1
\end{matrix} \right) \in \mathrm{GL}_2(K_\mathfrak{q}) = \mathrm{GL}_2(\mathbb{Q}_q) .$$ 
\item Let $q=p$ and suppose that $p = \mathfrak{p}\overline{\mathfrak{p}}$ splits in $K$. Then we put
$$\varsigma^{(n)}_p = \left( \begin{matrix}
\vartheta & -1 \\
1 & 0
\end{matrix} \right)
\left( \begin{matrix}
p^n & 0 \\
0 & 1
\end{matrix} \right)
 \in \mathrm{GL}_2(K_\mathfrak{p}) = \mathrm{GL}_2( \mathbb{Q}_{p} ).$$
\end{enumerate}
Suppose that $p$ is inert in $K$. Then we put
$$\varsigma^{(n)}_p = \left( \begin{matrix}
0 & 1 \\
-1 & 0
\end{matrix} \right)
\left( \begin{matrix}
p^n & 0 \\
0 & 1
\end{matrix} \right)
 \in \mathrm{GL}_2(K_p) = \mathrm{GL}_2( \mathbb{Q}_{p^2} ).$$

Now the the \textit{explicit Gross point} $\varsigma^{(n)}$ of conductor $p^n$ on $\widehat{B}^\times$ is defined as follows:
$$\varsigma^{(n)} := \varsigma^{(n)}_p \times \prod_{q \neq p} \varsigma_q \in \widehat{B}^\times .$$
\begin{remark}
    By the construction of Gross points and theta elements in \cite[\S~4.1]{ChidaHsieh2018}, the explicit Gross points $\varsigma^{(n)}$ coincide with the classical Gross points $v_n$   for the choice $\Psi:K\longrightarrow B$ (see \cite[\S~4.1]{Kim19}).
\end{remark}

\section{Ordinary case}\label{ord-case}
\subsection{$p$-adic $L$-function} 

In this section we recall the definition of $p$-adic $L$-function following \cite{Kim26}.
Let us assume that $\alpha,\beta$ be the roots of the Hecke polynomial $X^2-a_p(f)X+p$ of $f$ at $p$. If $f$ is ordinary at $p$, one of them say $\alpha$ is a $p$-adic unit. Assume that the residual Galois representation 
$\overline{\rho}_f$ attached to $f$ is irreducible. The $p$-stabilization $f_\alpha$ of $f$ is defined by 
\begin{equation*}
    f_\alpha(z)=f(z)-\beta\cdot f(pz).
\end{equation*}
The Bertolini--Darmon theta element of $f_\alpha$ over $K^{(n)}$ is characterized by the following relation
\begin{equation}\label{eq:rel}
    \cL_{f_\alpha,n}=\frac{1}{\alpha^n}(\cL_{f,n}-\frac{1}{\alpha}\cdot\xi_{n-1}\cL_{f,n-1}).
\end{equation}
By the three-term relation \eqref{three-term} we have 
\begin{equation*}
    \pi_{n+1}^n(\cL_{f_\alpha,n+1})=\cL_{f_\alpha,n}.
\end{equation*}
This tells us that the theta elements of $f_\alpha$ satisfy the norm compatibility relation.  
Then the anticyclotomic $p$-adic $L$-function of $f$ is defined by taking projective limit;  
\begin{equation*}
    L_p(f/K_\infty)=\varprojlim_n (\cL_{f_\alpha,n}\cdot \iota(\cL_{f_\alpha,n}))\in \Lambda,
\end{equation*}
where $\iota$ be the involution on $\Lambda_n$ sending $\gamma$ to $\gamma^{-1}$.
\begin{remark}\label{rmk:1}
    Recall that the involution map $\iota:\cO[[T]]\longrightarrow \cO[[T]]$ is defined by $1+T\mapsto (1+T)^{-1}$, and hence
$$
\iota(T)=\left(\frac{-1}{1+T}\right)T.
$$
Since $\frac{-1}{1+T}$ is a unit in $\cO[[T]]$, it follows that $\iota$ is an isomorphism. Hence, for any $P\in \cO[[T]]$, the Iwasawa invariants of $P$ and $P^\iota:=\iota(P)$ are the same.
\end{remark}

\subsection{Iwasawa invariants in the \texorpdfstring{$p$}{p}-ordinary case.}
Let $f$ continue to be a $p$-ordinary eigenform on $\Gamma=\Gal(K_\infty/K)$. 
By Lemma \ref{l3}, we have
\begin{equation}\label{invariants}
    2\mu(\cL_{f_\alpha,n})=\mu(L_p(f/K_\infty)) \quad \text{and} \quad 2\lambda(\cL_{f_\alpha,n})=\lambda(L_p(f/K_\infty)).
\end{equation}
Now we have the following theorem.
\begin{theorem}\label{main-thm-ord}
    Assume that the residual Galois representation 
    $\overline{\rho}_f$ attached to $f$ is irreducible and assume the hypotheses (\ref{disc}) and (\ref{def}). 
     Then for $n\gg 0$, we have 
    \begin{enumerate}
        \item $\mu(\cL_{f,n})=0$,
        \item $2\lambda(\cL_{f,n})=\lambda(L_p(f/K_\infty))$.
    \end{enumerate}
\end{theorem}

\begin{proof} 

The proof is based on \cite[Prop. 3.7]{PW} which uses Ihara's lemma \cite[Theorem 3.5]{PW}. 

By the work of Vatsal \cite{V} it is known that the $p$-adic $L$-function $L_p(f/K_\infty)$ is non-zero and  if $\overline{\rho}_f$ is irreducible, then one obtains that $\mu(L_p(f/K_\infty))=0$; see \cite[Theorem 2.3]{PW2011} for weight $2$ forms. 
This vanishing of $\mu$-invariant result is also known for weight $k$ (even) forms under some addition hypotheses, see \cite[Theorem 5.7]{ChidaHsieh2018}, the proof of which  uses an analogue of Ihara's lemma \cite[page 119, line 3]{ChidaHsieh2018} (see also  \cite[Theorem 6.8]{Manning_Shotton}). 

We will see from the argument below that the fact that $\mu(L_p(f/K_\infty))=0$ will be sufficient to deduce Theorem \ref{main-thm-ord}.

From \eqref{eq:rel}, we obtain
\begin{equation}\label{pstab}
    \cL_{f_\alpha,n}=\frac{1}{\alpha^n}(\cL_{f,n}-\frac{1}{\alpha}\xi_{n-1}\cL_{f,n-1}).
\end{equation}
By Lemma \ref{l3} we have $\mu(\cL_{f_\alpha,n})=0$ for $n\gg0$. 
Since $\alpha$ is a $p$-adic unit, it follows from  \eqref{pstab} that $\mu(\cL_{f,n})=0$, for sufficiently large $n$.
By Lemma \ref{l4} we have,
\begin{equation*}
    \lambda(\xi_{n-1} \cL_{f,n-1})\geq p^n-p^{n-1},
\end{equation*}
thus for large enough $n$ we have,
\begin{equation*}
    \lambda(\cL_{f_\alpha,n})<\lambda(\xi_{n-1} \cL_{f,n-1}).
\end{equation*}
Therefore, we obtain 
\begin{equation*}
\lambda(\cL_{f,n})=\lambda(\cL_{f_\alpha,n}).    
\end{equation*}

This shows that 
\begin{equation*}
   2\lambda(\cL_{f,n})=\lambda(L_p(f/K_\infty)). 
\end{equation*}

Now we know $\lambda(\cL_{f,n})\neq \lambda(\xi_{n-1} \cL_{f,n-1})$. This implies that the reduction of these two elements are not equal. Hence we have $\mu(\pi_{n+1}^n(\cL_{f,n+1}))=0$ by \eqref{three-term}. This implies $\mu(\cL_{f,n+1})=0$. This shows that $\mu(\cL_{f,n})=0$ for $n\gg0$. This completes the proof.

\end{proof}

\section{The non-ordinary case}\label{non-ord}
\subsection{Split case}

In this subsection, we assume that $N^-$ is square-free product of an odd number of primes, that $p$ is split in $K$. We begin by recalling the construction of logarithm matrices from \cite[\S~3.2]{BBL22}, which is based on \cite{spr12, spr17}.

\subsection{Construction of Sprung-type matrices}


For any integer $n\geq 0$, we write $\omega_n=\gamma^{p^n}-1\in \Lambda$ and  we denote the ring $\Lambda/(\omega_n)$ by $\Lambda_n$. Note that $\Lambda_n$ is isomorphic to $\cO[G_n]$ via the natural projection map. For $n\ge 1$, let $\Phi_n:=\frac{\omega_n}{\omega_{n-1}}$ be the $p^n$-th cyclotomic polynomial in the variable $\gamma$. For $n\ge0$, define the polynomials
\begin{equation*}
    \omega_n^+=\prod_{1\le j\le n \text{ even}} \Phi_j(\gamma), \quad \omega_n^-=\prod_{1\le j\le n \text{ odd}} \Phi_j(\gamma).
\end{equation*}


\noindent  For $n\geq 1$, define $M_{f, n}:=B_f^{-m-1}C_{f, n}\cdot\cdot\cdot C_{f, 1}$, where $C_{f, n}=\begin{pmatrix}a_p(f) & 1\\
-\Phi_{n} & 0
\end{pmatrix}\in M_2(\Lambda)$ and $B_f:=\begin{pmatrix}a_p(f) & 1\\
-p & 0
\end{pmatrix}\in M_2(\Lambda)$.
Note that $C_{f,n+1} \equiv B_f \mod \omega_n$, as we know $\Phi_{n}\equiv p\mod \omega_n$. This implies that the sequence $\{M_{f, n}\}_{n\in \N}$ forms a compatible system of matrices and converges to a matrix $M_{f, \log}\in M_2(\mathcal{H}(\Gamma))$.

 Consider the $\Lambda$-morphism 
\begin{equation*}
    H_{f,n}:\Lambda_n^2\longrightarrow \Lambda_n^2,
\end{equation*}
defined by 
\begin{equation*}
H_{f, n}\begin{pmatrix}x \\y \end{pmatrix}=C_{f, n}\cdot\cdot\cdot C_{f, 1}\begin{pmatrix}x \\y \end{pmatrix}. \\
\end{equation*}
Moreover, there is a natural isomorphism induced by the projections $\Lambda\longrightarrow \Lambda_n$, 
\[\Lambda^2\stackrel{\sim}{\longrightarrow}\varprojlim_{n}\Lambda^2/\ker(H_{f, n}). \]
Using the above isomorphism, we have the following lemma from \cite{BBL22}.  
\begin{lemma}\cite[Theorem~3.5]{BBL22}\label{lm:Hfn}
    There exist unique $\cL_f^\sharp, \cL_f^\flat\in \Lambda$ such that for $n\geq 1$ we have
\begin{equation*}
    H_{f, n}\begin{pmatrix}\cL_f^\sharp\\
\cL_f^\flat
\end{pmatrix}\equiv\begin{pmatrix}\cL_{f, n}\\ -\xi_{n-1}\cL_{f, n-1} \end{pmatrix}\mod \omega_n,
\end{equation*}
where $H_{f, n}$ is the logarithm matrix defined above as in \cite[Definition~3.3]{BBL22}.
\end{lemma}

For simplicity, we first consider the case when $a_p=0$. The case $a_p\ne 0$ will be dealt in a similar fashion with minor modification later in this section.  In this case we first compute the matrices $H_{f, n}$  for $n\geq 1$ as follows: 
\begin{equation}\label{eq:Hfn}
    H_{f, n}=\begin{cases}(-1)^{n/2}\begin{pmatrix}\omega^-_n & 0\\
0 & \omega^+_n\end{pmatrix} & \text{if } n \text{ is even}, \\
(-1)^{(n-1)/2}\begin{pmatrix}0 & \omega^+_{n}\\
-\omega^-_{n} & 0\end{pmatrix} & \text{if } n \text{ is odd}.
\end{cases}
\end{equation} 
  This observation gives us the following explicit description of $\cL_f^\sharp$ and $\cL_f^\flat$ in terms of $\cL_{f, n}$ and $\xi_{n-1}\cL_{f, n-1}$. Thus we have
\begin{equation}\label{eq:1}
    \begin{pmatrix}\cL_{f, n}\\ -\xi_{n-1}\cL_{f, n-1} \end{pmatrix}\equiv\begin{cases}(-1)^{n/2}\begin{pmatrix}\omega^-_n & 0\\
0 & \omega^+_n\end{pmatrix}\begin{pmatrix}\cL_f^\sharp\\
\cL_f^\flat
\end{pmatrix} \mod\omega_n & \text{if } n \text{ is even}, \\
(-1)^{(n-1)/2}\begin{pmatrix}0 & \omega^+_{n}\\
-\omega^-_{n} & 0\end{pmatrix}\begin{pmatrix}\cL_f^\sharp\\
\cL_f^\flat\end{pmatrix}\mod\omega_n & \text{if } n \text{ is odd}.
\end{cases}
\end{equation}
From the above observation, we deduce the following Lemma. 
\begin{lemma}\label{lm:2}
    For $n\geq 1$, 
    \begin{equation*}
    \cL_{f, n}\equiv\begin{cases}(-1)^{n/2}\omega^-_n\cL_f^\sharp \mod \omega_n & \text{if } n \text{ is even}, \\
(-1)^{(n-1)/2}\omega^+_{n}\cL_f^\flat \mod \omega_n & \text{if } n \text{ is odd}.
\end{cases}
    \end{equation*}
\end{lemma}
\begin{proof}
    This is clear from \eqref{eq:1}.
\end{proof}

 Recall the definition of $q_n$ from \cite{gajek}. 
\begin{equation*}
    q_n=\begin{cases}p^{n-1}-p^{n-2}+\cdots +p-1 & \text{if } n \text{ is even}, \\
p^{n-1}-p^{n-2}+\cdots +p^2-p & \text{if } n \text{ is odd}.
\end{cases}
\end{equation*} 
and $q_0=q_1=0.$
\begin{lemma}\label{lm:1}
For all $n\ge1$, we have $\mu(\omega_n^{\pm})=0$ and $\lambda(\omega_n^{\epsilon_{n+1}})=q_n$, where $\epsilon^{n+1}=(-1)^{n+1}$. 
\end{lemma}
\begin{proof}    The proof is follows from \cite[Lemma~3.7]{gajek}. 
\end{proof} 

\subsection{Iwasawa invariants in the non-ordinary case}
 To prove our main results in the non-ordinary case, we first prove the following propositions. 
\begin{Proposition}\label{main-thm}
   Assume the hypotheses (\ref{disc}) and (\ref{def}).  Let $f\in S_2(\Gamma_0(N))$ be a newform with $a_p(f)=0$. Assume that $\mu(\cL_f^\sharp)=\mu(\cL_f^\flat)\ne\infty$. Then for all sufficiently large $n$, we have 
     \[
\mu(\mathcal{L}_{f,n})=\mu(\mathcal{L}_f^\ast)
 \text{ and }
\lambda(\mathcal{L}_{f,n})=\lambda(\mathcal{L}_f^\ast)+q_n,
\]
where $\ast = \sharp$ if $n$ is even and $\ast = \flat$ if $n$ is odd. 
\end{Proposition}
\begin{proof}
Suppose $n\ge n_0^+$ is an even integer. Then from \cite[Corollary 2.7]{gajek} and lemma \ref{lm:2}, we have that $\mu(\cL_{f, n})=\mu(\omega^-_n\cL_f^\sharp)=\mu(\cL_f^\sharp)$ for sufficiently large $n$, since $\mu(\omega^+_n)=0$ (see lemma \ref{lm:1}). From Lemma \ref{lm:1}, we get that the $\lambda$-invariant of $\omega^-_n\cL_f^\sharp$ is equal to $q_n + \lambda(\mathcal{L}_f^\sharp)$. From the expression of $q_n$, it follows that $p^n-q_n\rightarrow \infty$ as $n\rightarrow\infty$ and therefore $\lambda(\cL_f^\ast)<p^n-q_n$,  for sufficiently large $n$.
Therefore by Lemma \ref{lm:1} and Lemma \ref{l3}, we have that 
$$\lambda(\cL_{f, n})=\lambda(\omega^-_n\cL_f^\sharp)=q_n+\lambda(\cL_f^\sharp)$$
for sufficiently large $n$. The proof for odd integer $n$ is similar. This completes the proof. 
\end{proof}

Now we consider the case $a_p\ne 0$. To prove the analogue of Proposition \ref{main-thm} for this case we need the following lemma. 
\begin{lemma}
     For $n\geq 1$, 
    \begin{equation*}
    \cL_{f, n}\equiv\begin{cases}(-1)^{n/2}\omega^-_n\cL_f^\sharp \mod (\omega_n,\varpi)  & \text{if } n \text{ is even}, \\
(-1)^{(n-1)/2}\omega^+_{n}\cL_f^\flat \mod (\omega_n, \varpi) & \text{if } n \text{ is odd}.
\end{cases}
    \end{equation*}
\end{lemma}
\begin{proof}
    We first note that $a_p\equiv 0\pmod{\varpi}$. Therefore,
    \begin{equation*}
        H_{f, n}\equiv\begin{cases}(-1)^{n/2}\begin{pmatrix}\omega^-_n & 0\\
0 & \omega^+_n\end{pmatrix} \pmod{\varpi} & \text{if } n \text{ is even}, \\
(-1)^{(n-1)/2}\begin{pmatrix}0 & \omega^+_{n}\\
-\omega^-_{n} & 0\end{pmatrix}\pmod{\varpi} & \text{if } n \text{ is odd}.
\end{cases}
\end{equation*}
and hence the proof follows from lemma \ref{lm:1}. 
\end{proof}
Now we have the following analogue of Theorem \ref{main-thm} for the case $a_p\ne 0$.  
\begin{Proposition}\label{thm-ap_non_0}
  Assume the hypotheses (\ref{disc}) and (\ref{def}).  Let $f\in S_2(\Gamma_0(N))$ be a newform with non-ordinary at the prime $p$.  Assume that $\mu(\cL_f^\sharp)=\mu(\cL_f^\flat)\ne\infty$. Then for all sufficiently large $n$, we have 
     \[
\mu(\mathcal{L}_{f,n})=\mu(\mathcal{L}_f^\ast)
 \text{ and }
\lambda(\mathcal{L}_{f,n})=\lambda(\mathcal{L}_f^\ast)+q_n,
\]
where $\ast = \sharp$ if $n$ is even and $\ast = \flat$ if $n$ is odd. 
\end{Proposition}
\begin{proof}
    Using Lemma \ref{lm:3.5}, the proof is exactly analogues to the proof of Proposition \ref{main-thm}.
\end{proof}

Let $\alpha$ and $\beta$ are roots of the Hecke polynomial $X^2-a_p(h)X+p$. For any $\lambda\in \set{\alpha, \beta}$, the \emph{$\lambda$-stabilized Bertolini--Darmon element} is defined by
\[\cL_{h, n}^\lambda:=\frac{1}{\lambda^{n+1}}\left(\cL_{h, n}-\frac{1}{\lambda}\xi_{n-1}\cL_{h, n-1}\right). \]
It follows that from the three-term relation \eqref{three-term} one obtains that $\pi_{n+1,n}(\cL_{h, n+1}^\lambda)=\cL_{h, n}^\lambda$ for $n\geq 0$. The sequence $\{\cL_{h, n}^\lambda\}_{n\geq 0}$ is compatible with the natural projections $\pi_{n+1,n}$. This sequence converges to an element $\cL_h^\lambda=\varprojlim \cL_{h, n}^\lambda\in \mathcal{H}[\Gamma]$, where $\mathcal{H}(\Gamma)$ denotes the set of power series in $\cK[[T]]$ that converges in the open unit disk.  The following result is related to the non-vanishing  of $p$-adic $L$-function. 
\begin{lemma}
    For $\lambda\in \set{\alpha, \beta}$, $\cL_f^\lambda\ne 0.$
\end{lemma}
\begin{proof}
    See \cite[Lemma~3.8]{BBL22}.
\end{proof} 
Using the above result it is shown in \cite[Theorem~3.9]{BBL22} that at least one of the two elements $\cL_f^\sharp$ and $\cL_f^\flat$ is non zero. Moreover, if $a_p=0$, then both $\cL_f^\sharp$ and $\cL_f^\flat$ are non-zero (see the proof of \cite[Theorem~3.9]{BBL22}). 

Following \cite{BBL22}, we define the following $p$-adic $L$-function. 
\begin{definition}\cite[Definition~3.6]{BBL22}\label{Def-L_function}
  The $p$-adic $L$-function associated to $f$ is defined as follows:
  \[L_p^\ast(f, K):=\cL_f^\ast(\cL_f^\ast)^\iota\]
  for $\ast\in\set{\sharp, \flat}$, where $\iota$ is the involution map on $\Lambda$. 
\end{definition}

 The following result is the relation between Iwasawa invariants of $p$-adic $ L$- functions and Iwasawa invariants of theta elements.  
 \begin{theorem}\label{main-thm-2}
 Assume the hypotheses (\ref{disc}) and (\ref{def}).  Let $f\in S_2(\Gamma_0(N))$ be a newform with $a_p(f)\equiv0\mod{\varpi}$. Assume that $\mu(L_p^\sharp(f, K_\infty))=\mu(L_p^\flat(f, K_\infty))\ne\infty$.
      Then for sufficiently large n, we have 
     \[
\mu(L_p^\ast(f, K))=2\mu(\mathcal{L}_{f,n})
 \text{ and }
2\lambda(\mathcal{L}_{f,n})=2q_n+\lambda(L_p^\ast(f,K))
\]
where $\ast = \sharp$ if $n$ is even and $\ast = \flat$ if $n$ is odd.  
 \end{theorem}
\begin{proof}
   By Remark \ref{rmk:1}, we have that the Iwasawa invariants remain unchanged under the involution map $\iota$. Therefore from \cite[Lemma~2.2]{gajek}, we have that
    \[\mu(L_p^\ast(f, \Qp))=2\mu^\ast(f),\]
    for $\ast\in\set{\sharp, \flat}$. 
    
For $\lambda$-invariants, note that $\lambda(L_p^\ast(f, K))=2\lambda(\cL_f^\ast)$. It follows from Lemma \ref{l3} together with  Proposition \ref{main-thm},  Proposition \ref{main-thm-2} that,  $\lambda(L_p^\ast(f, K))=2\lambda(\cL_f^\ast)=2(\lambda(\cL_{f, n})-q_n)$ for sufficiently large $n$. This completes the proof. 
\end{proof}

\subsection{Inert case}\label{inert_sec}
In this section, we assume that $p\ge5$ and $p$ is inert in $K$. Let $E/\Q$ be an elliptic curve and $f:=f_E\in S_2(\Gamma_0(N))$ be the corresponding modular form. Assume that $f$ is supersingular at $p$, i.e. $a_p(f)=0$. In this case, we follow the setup of \cite[\S~4.3]{BLV}. We denote $\Lambda$ for the Iwasawa algebra $\Zp[[T]]$ and $\Lambda_n=\Lambda/(\omega_n)=\Zp[G_n]$, where $\omega_n$ is as before. Define
\begin{enumerate}
    \item $\tilde{\omega}_0^+=\tilde{\omega_1}^+=1.$
    \item $\tilde{\omega}_0^-=\gamma-1$, where $\gamma$ is a topological generator of $G_\infty$. 
    \item For $n\ge 2, $ 
    $$\tilde{\omega}_n^+=\prod_{1\le j\le n \text{ even}} \Phi_j(\gamma),$$
    
    \item For $n\ge 1,$ $$\tilde{\omega}_n^-=\prod_{1\le j\le n \text{ odd}} \Phi_j(\gamma).$$
\end{enumerate}
Note that $\omega_n^+=\tilde{\omega}_n^+$ and $\omega_n^-=(\gamma-1)\tilde{\omega}_n^-$.   Set $\Lambda_{n}^\pm:=\dfrac{\Lambda}{(\omega_n^\pm)}$. 

 \begin{lemma}\label{lm:1:inert}
For all $n\ge1$, we have $\mu(\tilde{\omega}_n^{\pm})=0$, $\lambda({\omega}_n^+)=q_n$ if $n$ is odd and $\lambda({\omega}_n^-)=q_n+1$ if $n$ is even. 
\end{lemma}  
\begin{proof} The proof is similar to the  Lemma \ref{lm:1} and  \cite[Lemma~3.7]{gajek} using the fact that  $\mu(\gamma-1)=0$ and $\lambda(\gamma-1)=1$. 
\end{proof}
Recall that the projection map $\pi_{n+1}^n:\Zp[G_{n+1}]\longrightarrow \Zp[G_n]$ is defined by $\pi_{n+1}^n(F_{n+1}):= F_{n+1}\mod \omega_n$.   
\begin{lemma}\label{lm:intert}
  Let $\epsilon=(-1)^n$, then $\omega_n^\epsilon\cL_{f, n}\in \omega_n\cdot\Lambda$ for all integers $n\ge 0$. 
  \end{lemma}
  \begin{proof}
      The proof is similar to \cite[Lemma~4.3]{BLV}.

      We prove this by induction on $n$. For $n=1$, the result is trivially true because $\omega_1^-=\omega_1$. Suppose the result is true for all integers $t$ with $0\le t\le n$ and then we prove it for $n+1$.
      
      Let $\gamma$ be a generator of $G_n$, and let its image in $\Zp[G_n]$ also be denoted by $\gamma$. By a slight abuse of notation, we shall use the same symbol $\gamma$ to denote any lift of $\gamma$ to $\Zp[G_{n+1}]$. From the three-term relation and the definition of the projection map, we have the following:
      \[
      \cL_{f,n+1}\equiv -\xi_n\cL_{f, n-1}\mod\omega_n. 
      \]
Now for $n=2m+2$, we have that 
\begin{align*}
\omega_{2m+2}^{+}\cdot \mathcal{L}_{f,2m+2}
&=
\omega_{2m+2}^{+}\cdot
\left(
-\xi_{2m+1}\cdot \mathcal{L}_{f,2m}
+\omega_{2m+1}\cdot \alpha
\right)
\qquad
(\text{for some }\alpha\in\Lambda_{2m+2})
\\
&=
-\xi_{2m+2}(\gamma)\cdot \xi_{2m+1}(\gamma)\cdot
\omega_{2m}^{+}\cdot \mathcal{L}_{f,2m}
+
\omega_{2m+2}^{+}\cdot \omega_{2m+1}\cdot \alpha
\\
&\in
\underbrace{
-\xi_{2m+2}(\gamma)\cdot
\xi_{2m+1}(\gamma)\cdot
\omega_{2m}\cdot
\Lambda_{2m+2}
}_{=\omega_{2m+2}}
+
\underbrace{
\omega_{2m+2}^{+}\cdot
\omega_{2m+1}\cdot
\Lambda_{2m+2}
}_{\text{divisible by }\omega_{2m+2}}.
\end{align*}   

Note that the last inequality follows from the induction hypothesis as
$\omega_{2m+2}^{+}\omega_{2m+1}$ is divisible by $\omega_{2m+2}$,
since $\omega_{2m+2}^{-}=\omega_{2m-1}^{-}$. For $n=2m+1$, a similar argument gives the desired result. 
\end{proof}

In the view of \cite[Lemma~4.4]{BLV}, 
 multiplication by $\omega_n^+$ is an isomorphism from $\Lambda_n^-$ to $\omega_n^-$-torsion submodule of $\Lambda_n$, and multiplication by $\omega_n^-=(\gamma-1)\tilde{\omega}^-_n$ is an isomorphism from $\Lambda^+_n$ to $\omega^+_n=\tilde{\omega}^+_n$-torsion submodule of $\Lambda_n$.  From the above observation, there exist $\cL_{f,n}^\pm\in \Lambda^{\pm}_{n}$ such that
for $n\geq 1$, 
    \begin{equation}\label{non-exp}
    \cL_{f, n}=\begin{cases}(-1)^{n/2}{\omega}^-_n\cL_{f,n}^+ & \text{if } n \text{ is even}, \\
(-1)^{(n-1)/2}\tilde{\omega}^+_{n}\cL_{f, n}^- & \text{if } n \text{ is odd}.
\end{cases}
    \end{equation}

Denote by $\pi^+_{2m+2} : \Lambda^{+}_{2m+2} \to \Lambda^{+}_{2m}$ and $\pi^{-}_{2m+3} : \Lambda^{-}_{2m+3} \to \Lambda^{-}_{2m+1}$ the natural projection maps. Then, from \cite[Lemma~4.5]{BLV}, it follows that $\cL_{f, 2m}^+$ (resp. $\cL_{f, 2m+1}^-$) forms a compatible system with respect to the projection maps $\pi^{+}_{2m}$ (resp. $\pi^{-}_{2m+1}$) for every $m\ge1$. Note that $\varprojlim \Lambda^+_{2m}\cong\Lambda\cong\varprojlim \Lambda^-_{2m+1}$, then one can define 

\[\cL_f^{\pm}=\underset{n\in\N^\pm}{\varprojlim}\cL^\pm_{f, n}\in \Lambda,  \]
where $\N^+$ (resp. $\N^-$) be the set of even (resp. odd) positive integers.
The $p$-adic $L$-function is defined as follows:
\[L_p^\pm(f, K):=\cL_f^\pm(\cL_f^\pm)^\iota.\]
\begin{lemma}\label{lm:inert}
For $n\ge 1$, we have  
\begin{equation*}
    \cL_{f, n}\equiv\begin{cases}(-1)^{n/2}{\omega}^-_n\cL_{f}^+\mod{\omega_n} & \text{if } n \text{ is even}, \\
(-1)^{(n-1)/2}\tilde{\omega}^+_{n}\cL_{f}^- \mod{\omega_n}& \text{if } n \text{ is odd}.
\end{cases}
    \end{equation*}
\end{lemma}
\begin{proof}
    From the definition of $\cL_f^\pm$, it follows that $\cL_f^\pm\equiv \cL_{f, n}^\pm\mod{\omega_{n}^\pm}$.
    Recall that multiplication by $\tilde{\omega}_n^{+}$ (resp. $\omega_n^{-}=(\gamma-1)\widetilde{\omega}_n^{-}$) induces an isomorphism
\[
\Lambda_{n}^{-}\xrightarrow{\sim}\Lambda_{n}[\omega_n^{-}]
\quad
\left(\text{resp. } \Lambda_{n}^{+}\xrightarrow{\sim}\Lambda_{n}[\omega_n^{+}]
=\Lambda_{n}[\widetilde{\omega}_n^{+}]\right).
\]
This gives us  $\omega_n^\mp\cL_f^\pm\equiv \omega_n^\mp \cL_{f, n}^\pm \mod{\omega_n}$. Using this observation and \eqref{non-exp}, the lemma follows. 
\end{proof}

Now we have the following theorem in the inert case. 
\begin{theorem}\label{main_thm_inert}
Assume the hypotheses (\ref{disc}) and (\ref{def}).  Let $f\in S_2(\Gamma_0(N))$ be a newform with $a_p(f)=0$. Assume that $\mu(L_p^+(f, K_\infty))=\mu(L_p^-(f, K_\infty))<\infty$.  Then for sufficiently large $n$, we have the following: 
\[\mu(L_p^\pm(f, K))=2\mu(\mathcal{L}_{f,n})
 \text{ and }
2\lambda(\mathcal{L}_{f,n})=\begin{cases}
  2(q_n+1)+\lambda(L_p^+(f,K)) & \text{if $n$ is even}, \\
  2q_n+\lambda(L_p^-(f,K))& \text{if $n$ is odd}. 
\end{cases}
\]

 \end{theorem}
\begin{proof} 
    Note that for odd $n$, the proof is identical to that of  the split case (see Theorem \ref{main-thm-2}). Suppose $n$ is even, then Lemma \ref{l3} and Lemma \ref{lm:inert} gives us that $\mu(\lambda(L_p^+(f,K)))=2\mu(\cL_f^+)=2\mu(\cL_{f, n})$ for sufficiently large $n$. Using an argument similar to Proposition \ref{main-thm}, we compute the $\lambda$-invariants.  Note that,  Lemma \ref{l3} gives us that $\lambda(L_p^+(f,K))=2\lambda(\cL_f^+)$. It follows from Lemma \ref{lm:1:inert} together with Lemma \ref{lm:inert} that $\lambda(L_p^+(f,K))=2(q_n+1)-\lambda(\cL_{f, n})$, for sufficiently large $n$. This completes the proof. 
    
\end{proof}

\bibliographystyle{plain}
\bibliography{references}

\end{document}